\documentclass[12pt,a4paper]{article}
\usepackage{amsmath,amsthm,amsfonts,amssymb,bbm}
\usepackage{graphicx,psfrag,subfigure,color}
\usepackage{cite}
\usepackage{hyperref}

\numberwithin{equation}{section}

\newcommand{\LL}{{{\cal L}_\rho}}
\newcommand{\Lflat}{{{\cal L}^{\rm flat}_\rho}}
\newcommand{\Or}{\mathcal{O}}
\newcommand{\Ai}{\mathrm{Ai}}

\newcommand{\Pb}{\mathbb{P}}
\newcommand{\E}{\mathbbm{E}}
\newcommand{\Id}{\mathbbm{1}}
\newcommand{\e}{\varepsilon}

\newcommand{\R}{\mathbb{R}}

\newcommand{\Z}{\mathbb{Z}}

\newtheorem{prop}{Proposition}[section]
\newtheorem{thm}[prop]{Theorem}
\newtheorem{lem}[prop]{Lemma}

\newtheorem{cor}[prop]{Corollary}

\newtheorem{cla}[prop]{Claim}

\newtheorem{rem}[prop]{Remark}
\newenvironment{remark}{\begin{rem}\normalfont}{\end{rem}}

\title{Universality of the GOE Tracy-Widom distribution for TASEP with arbitrary\\ particle density}
\author{P.L. Ferrari\thanks{Institute for Applied Mathematics, Bonn University, Endenicher Allee 60, 53115 Bonn, Germany. E-mail: {\tt ferrari@uni-bonn.de}}
\and  A. Occelli\thanks{Institute for Applied Mathematics, Bonn University, Endenicher Allee 60, 53115 Bonn, Germany. E-mail: {\tt occelli@iam.uni-bonn.de}}}
\date{April 24, 2018}

\begin{document}
\sloppy
\maketitle

\begin{abstract}
We consider TASEP in continuous time with non-random initial conditions and arbitrary fixed density of particles $\rho\in (0,1)$. We show GOE Tracy-Widom universality of the one-point fluctuations of the associated height function. The result phrased in last passage percolation language is the universality for the point-to-line problem where the line has an arbitrary slope.
\end{abstract}

\section{Introduction}
We consider the totally asymmetric simple exclusion process (TASEP) in continuous time on $\Z$. It is an interacting particle system with the constraint that there is at most one particle per site. Particles jump to their right-neighboring site with rate $1$, provided the arrival site is empty. A very natural and important observable is the integrated current at (for example) the origin, that is,
\begin{equation}
J(t)=\#\,\textrm{ particles which jumped from site }0\textrm{ to site }1\textrm{ during time }[0,t].
\end{equation}

TASEP is a model in the Kardar-Parisi-Zhang (KPZ) universality class and thus one expects that for some model-dependent constants, $c_1,c_2$,
\begin{equation}
t\mapsto \frac{J(t)-c_1 t}{c_2 t^{1/3}}
\end{equation}
has in the $t\to\infty$ limit a non-trivial distribution function, say $D$. It is well-known that for KPZ models the distribution $D$ depends on classes of initial conditions~\cite{BR99,BR99b,PS00} (see also the reviews~\cite{Fer07,Cor11}). In particular, consider the case of non-random initial condition with density $\rho=1/2$, realized by placing at time $0$ particles on every even sites. The joint distribution of the current at different points has been studied~\cite{Sas05,BFPS06}. As a particular case, the one-point distribution is given by the Fredholm determinant, which is shown to be equal to the GOE Tracy-Widom distribution in~\cite{FS05b},
\begin{equation}
\lim_{t\to\infty} \Pb\left(J(t)\geq \tfrac14 t - s 2^{-2/3} t^{1/3}\right)=F_{\rm GOE}(2^{2/3} s),
\end{equation}
where $F_{\rm GOE}$ denotes the GOE Tracy-Widom distribution function discovered first in random matrix theory~\cite{TW96}. The analogue result was previously known for discrete time TASEP with parallel update and for a combinatorial model of longest increasing subsequences with involutions~\cite{BR99,BR99b}. This latter model was brought in connection to the KPZ world in~\cite{PS00}, where it was reinterpreted as a stochastic growth model (the so-called polynuclear growth model).

From~\cite{Jo03} we also have the variational formula
\begin{equation}\label{eq1}
F_{\rm GOE}(2^{2/3} s)=\Pb\Big(\max_{v\in\R} \{{\cal A}_2(v)-v^2\}\leq s\Big),
\end{equation}
where ${\cal A}_2$ is called the Airy$_2$ process~\cite{PS02,Jo03b}. There are many more variational formulas related with the Airy$_2$ process, see e.g.~\cite{BL13} and the review~\cite{QR13}.

By universality one expects that the GOE Tracy-Widom distribution describes the fluctuations of $J(t)$ in the large time limit for any non-random initial condition with density $\rho\in(0,1)$. Beyond the case of $\rho=1/2$, this was proven for densities $\rho=1/d$, $d=2,3,4,\ldots$ in~\cite{BFP06}, and for the low-density limit of reflecting Brownian motions in~\cite{FSW13} (in these works also the joint distribution of the current have been analyzed). In these papers, the results are achieved by exact formulas for a correlation kernel which describes the system. However, beyond the $d=2$ case, the asymptotic analysis in these special cases turned out to be quite involved. An exact formula has very recently been derived for arbitrary initial condition as well~\cite{MQR17}. Formulas for the system with periodic boundary condition are also know only for densities $1/2,1/3,\ldots$~\cite{BL16,BL16b}.

In this paper we prove that for any $\rho\in (0,1)$,
\begin{equation}
\lim_{t\to\infty} \Pb\left(J(t)\geq \rho(1-\rho) t - s (\rho(1-\rho))^{2/3} t^{1/3}\right)=F_{\rm GOE}(2^{2/3} s);
\end{equation}
compare this with Corollary~\ref{CorMainTASEP}. The proof of our result is in his core probabilistic, where the only input from exactly solvable cases is the convergence to the Airy$_2$ process for the so-called step initial condition and bounds on the tails of its one-point distribution. We prove the convergence to the variational problem (\ref{eq1}), which does not depend on $\rho$. For $\rho=1/2$ the limiting distribution function was already known to be given by $F_{\rm GOE}$. The method allows for more general, including random initial conditions, we first prove convergence to a more generic variational process in Theorem~\ref{ThmMainTASEP}.

To show the convergence to the variational problem, we work in the last passage percolation (LPP) framework (see Section~\ref{sec:TASEPandLPP} for definitions and details). In that language we need to study a ``line-to-point'' problem with the line having arbitrary slopes. Using a tightness result for the ``point-to-point'' problem (see Theorem~\ref{ThmTightness}) and a slow-decorrelation result (see Theorem~\ref{ThmSlowDecOnePt}) (which is then extended to a functional slow-decorrelation theorem (see Theorem~\ref{ThmFctSlowDec})) we can show, analogously to~\cite{CLW16}, the convergence of a restricted ``line-to-point'' LPP problem to the variational problem (\ref{eq1}) with $|u|\leq M$. The second step of the proof consists in showing that the original LPP is localized, which is obtained by obtaining a bound on the probability that the maximizer of the LPP is not localized on a $\Or(M t^{2/3})$ region. In particular, for the flat initial condition case, we obtain a Gaussian bound in $M$, see Lemma~\ref{LemmaBoundOnR} (for an analogue bound on the limit process, see Proposition~4.4 of~\cite{CH11}).

The strategy to prove the convergence for the restricted was first developed by Corwin, Liu and Wang in~\cite{CLW16}. In that paper, for generic initial conditions (possibly random) they obtained universal results showing that the distribution converges to a variational problem (which depends on how the initial condition scales under diffusive scaling), for cases which are macroscopically at density $1/2$.
In the continuous time setting, this was studied in~\cite{CFS16}. In particular, if the initial condition ``scales subdiffusively'', then for $\rho=1/2$ one still sees $F_{\rm GOE}$ fluctuations. This fact was predicted in the context of the KPZ equation in~\cite{QR16}.

The main technical novelty of our proof concerns the localization. In particular, unlike in~\cite{CLW16,CFS16}, we do not require any extra input from solvable models beyond the ones which are used to prove convergence in the restricted LPP problem. All we need is a good control on the point-to-point process along a horizontal line. The key idea is to bound the increment of the process by the ones of two stationary initial conditions, with densities slightly higher/lower than $\rho$, which are chosen such that the inequality holds on a set of high probability. This probability is given in terms of some exit point probabilities. This comparison was used first by Cator and Pimentel in~\cite{CP15b} (see also~\cite{Pim17}) to show tightness for the Hammersley process and the point-to-point LPP along a characteristic direction with ''speed'' $0$. In Lemma~\ref{lemmaExitPoint} we obtain much stronger exit point probabilities than in~\cite{Pim17}. More importantly, we use the inequality in two ways: (a) to extend the tightness result to any characteristic direction (which is needed to the analysis any density $\rho$), and (b) to control the fluctuations of the process over large distances (of order $M t^{2/3}$).

The control of the fluctuations over large distances is indeed a key ingredient to obtain the localization bound. This reduces the input from exactly solvable models with respect to~\cite{CLW16,CFS16}. In~\cite{CLW16} they introduced a non-intersecting line ensemble and the bound followed using its Gibbs-Brownian property in a smart way. In~\cite{CFS16} the bound was obtained using an explicit correlation kernel for the so-called ''half-flat'' initial condition. This approach allowed to simplify~\cite{CLW16}, but it has the drawback that it is restricted to the case $\rho=1/2$.

The main problem in analyzing directly $\rho\not\in\{1/2,1/3,1/4,\ldots\}$ was that an explicit expression for the correlation kernel was not known. In the recent paper on KPZ fixed point by Matetski, Quastel and Remenik~\cite{MQR17} they found an explicit representation of it which could be used to obtain our result (and also the convergence to the Airy$_1$ process). However, the analysis has been made only for $\rho=1/2$, since it was enough for answering the question on the KPZ fixed-point considered in the paper.

Although the method in this paper allows to get convergence only for the one-point distribution, its strategy could be used also for other models in the KPZ universality class. For instance, for the partially asymmetric simple exclusion process (PASEP), where an analogue of the work~\cite{MQR17} seems out of reach (an exact formula allowing the asymptotic analysis for PASEP even with $\rho=1/2$ is not known, although heavy efforts have been made in particular by Ortmann, Quastel and Remenik~\cite{OQR15,OQR16}). On the other hand, ingredients like slow-decorrelation hold also for PASEP using basic coupling~\cite{CFP10b}. Furthermore, as shown in~\cite{Fer17}, the mapping to LPP is actually not needed to analyze TASEP. This observation is relevant since for PASEP this mapping does not exist anymore. The main missing ingredient for an extension to PASEP is the convergence to the Airy$_2$ process for step initial condition. This is an open problem, but it looks easier than the analysis of PASEP with general densities $\rho$ through exact formulas (compare with the formulas for $\rho=1/2$ of~\cite{OQR15,OQR16}).

\bigskip
\emph{Outline}. In Section~\ref{sect:main} we define TASEP, LPP and present the main results. Section~\ref{sect:tightness} contains the proof of tightness and the derivation of a bound needed to control localization as well. Finally, we prove the main theorem for LPP and TASEP in Section~\ref{sect:mainThmproof}.

\bigskip\noindent
{\bf Acknowledgments.} The work is supported by the German Research Foundation as part of the SFB 1060--B04 project.

\section{Main results}\label{sect:main}

\subsection{LPP and TASEP} \label{sec:TASEPandLPP}
A last passage percolation (LPP) model on $\Z^2$ with independent random variables $\{\omega_{i,j},i,j\in\Z\}$ is the following. An \emph{up-right path} $\pi=(\pi(0),\pi(1),\ldots,\pi(n))$ on $\Z^2$ from a point $A$ to a point $E$ is a sequence of points in $\Z^2$ with \mbox{$\pi(k+1)-\pi(k)\in \{(0,1),(1,0)\}$}, with $\pi(0)=A$ and $\pi(n)=E$, and where $n$ is called the length $\ell(\pi)$ of $\pi$. Now, given a set of points $S_A$ and $E$, one defines the last passage time $L_{S_A\to E}$ as
\begin{equation}\label{eq3.2}
L_{S_A\to E}=\max_{\begin{subarray}{c}\pi:A\to E\\A\in S_A\end{subarray}} \sum_{1\leq k\leq \ell(\pi)} \omega_{\pi(k)}.
\end{equation}
Finally, we denote by $\pi^{\rm max}_{S_A\to E}$ any maximizer of the last passage time $L_{S_A\to E}$. For continuous random variables, the maximizer is a.s.\ unique.

TASEP is an interacting particle system on $\Z$ with state space $\Omega=\{0,1\}^\Z$. For a configuration $\eta\in\Omega$, $\eta=(\eta_j,j\in\Z)$, $\eta_j$ is the occupation variable at site $j$, which is $1$ if and only if $j$ is occupied by a particle. TASEP has generator $L$ given by~\cite{Li99}
\begin{equation}\label{1.1}
Lf(\eta)=\sum_{j\in\Z}\eta_j(1-\eta_{j+1})\big(f(\eta^{j,j+1})-f(\eta)\big),
\end{equation}
where $f$ are local functions (depending only on finitely many sites) and $\eta^{j,j+1}$ denotes the configuration $\eta$ with the
occupations at sites $j$ and \mbox{$j+1$} interchanged. Notice that for the TASEP the ordering of particles is preserved. That is, if initially one orders from right to left as
\[\ldots < x_2(0) < x_1(0) < 0 \leq x_0(0)< x_{-1}(0)< \cdots,\]
then for all times $t\geq 0$ also $x_{n+1}(t)<x_n(t)$, $n\in\Z$.

TASEP can be also though as a growth process by introducing the height function $h(j,t)$ as
\begin{equation}\label{1.11}
h(j,t)=
 \begin{cases}
 2J(t) +\sum^j_{i=1}(1-2\eta_i(t)) & \textrm{for }j\geq 1,\\
 2J(t) & \textrm{for }j=0,\\
 2J(t) -\sum^0_{i=j+1}(1-2\eta_i(t)) & \textrm{for }j\leq -1,
 \end{cases}
\end{equation}
for $j\in\Z$, $t\geq 0$, where $J(t)$ counts the number of jumps from site $0$ to site $1$ during the time-span $[0,t]$.

The connection between TASEP and LPP is as follows. Take $\omega_{i,j}$ to be the waiting time of particle $j$ to jump from site $i-j-1$ to site $i-j$.
Then $\omega_{i,j}$ are ${\rm Exp}(1)$ i.i.d.\ random variables. Further, setting the set $S_A=\{(u,k)\in\Z^2: u=k+x_k(0), k\in\Z\}$, we have that
\begin{equation}\label{eq2.4}
\Pb\left(L_{{S_A}\to (m,n)}\leq t\right)=\Pb\left(x_n(t)\geq m-n\right)=\Pb\left(h(m-n,t)\geq m+n\right).
\end{equation}

\subsection{Universality for LPP}
For any fixed $\rho\in (0,1)$, we consider the LPP model with $S_A$ corresponding to TASEP with initial condition $x_k^{\rm flat}(0)=-\lfloor k/\rho\rfloor$, $k\in\Z$. We denote this initial set by
\begin{equation}\label{eqL}
\Lflat=\left\{\left(\lfloor \tfrac{\rho-1}{\rho} x\rfloor,x\right),x\in\Z\right\}
\end{equation}
and we are interested in the LPP from $\Lflat$ to $E_N(w)$ in the limit $N\to\infty$ illustrated in Figure~\ref{FigLPP}. However, the approach used in the proof allows to consider more general (also random) initial conditions. Thus we consider TASEP with initial condition close to the flat initial condition with density $\rho$ as well. Denote by
\begin{equation}
u_k=x_k(0)-x_k^{\rm flat}(0)
\end{equation}
the deviation of the particle position with respect to the flat initial condition with density $\rho$. In this setting, in the LPP setting, we need to consider the initial set
\begin{equation}\label{eqLgeneral}
\LL=\left\{\left(\lfloor\tfrac{\rho-1}{\rho} k\rfloor+u_k,k\right),k\in\Z\right\}.
\end{equation}
We also denote
\begin{equation}
\chi=\rho(1-\rho).
\end{equation}
Let
\begin{equation}\label{eqDefSflat}
A^{\rm flat}(v)=\left(-2(1-\rho)\chi^{-1/3}v N^{2/3},2 \rho\chi^{-1/3}v N^{2/3}\right)
\end{equation}
and define by $A(v)$ the closest point on $\LL$ to the characteristic line with direction $\mathbf{e}_\rho=((1-\rho)^2,\rho^2)$ passing by $A^{\rm flat}(v)$. Then define $\lambda(v)$ by
\begin{equation}\label{eqDefS}
A(v) = A^{\rm flat}(v)+\lambda(v)\mathbf{e}_\rho
\end{equation}
To avoid that the randomness in the initial condition dominates the bulk ones, we assume\\
\textbf{Assumption A:}
\begin{equation}
\lim_{N\to\infty} \frac{\lambda(v)}{\chi^{-2/3} N^{1/3}}= {\cal R}(v)=\sqrt{2}\sigma {\cal B}(v),
\end{equation}
weakly on the space of continuous functions on bounded sets, where ${\cal B}$ is a two-sided Brownian motion and $\sigma\geq 0$ a coefficient. The stationary initial condition is $\sigma=1$, while the flat initial condition is $\sigma=0$.

Furthermore, we assume that globally the starting height function (or particle positions) are not deviating too much from the flat case, so that the maximization problem is non-trivially correlated only with the randomness in a $N^{2/3}$-neighborhood of the origin. \\
\textbf{Assumption B:} For any given $\delta>0$ and $M>0$, there exists a $N_0$ such that for all $N\geq N_0$,
\begin{equation}
\Pb(\lambda(v)\geq -\delta v^2 N^{1/3}\textrm{ for all }|v|\geq M)\geq 1-Q(M),\quad \lim_{M\to\infty} Q(M)=0,
\end{equation}
where $v$ are restricted to those such that $A(v)$ is connected to the end-point of the LPP by an up-right path.

These assumptions clearly holds for LPP corresponding to flat initial conditions, but also to the case where the deviation of the initial height function scales diffusively like in the stationary initial conditions. Under these assumptions we show the following universality result.
\begin{thm}\label{ThmMainLPP}
Let $\rho\in (0,1)$, $\chi=\rho(1-\rho)$. Set the end-point of the LPP as $E_N(w)=(m_N(w),n_N(w))$ with
\begin{equation}
\begin{aligned}
m_N(w)&=\tfrac{1-\rho}{\rho} N-2 w (1-\rho)\chi^{-1/3}N^{2/3},\\
n_N(w)&=\tfrac{\rho}{1-\rho} N +2 w \rho \chi^{-1/3}N^{2/3},
\end{aligned}
\end{equation}
Under Assumptions A and B, for any $s\in\R$,
\begin{equation}\label{eqMainThm}
\lim_{N\to\infty} \Pb\left(L_{\LL\to E_N(w)}\leq \frac{N}{\chi} + \frac{s N^{1/3}}{\chi^{2/3}}\right) = \Pb\left(\max_{v\in\R} \{{\cal A}_2(v)-(v-w)^2+{\cal R}(v)\}\leq s\right).
\end{equation}
where ${\cal A}_2$ is the Airy$_2$ process~\cite{PS02}.
In particular, for LPP from $\Lflat$, for which ${\cal R}=0$, we have
\begin{equation}\label{eqMainThmBis}
\lim_{N\to\infty} \Pb\left(L_{\Lflat\to E_N(w)}\leq N/\chi + s N^{1/3}/\chi^{2/3}\right) =  F_{\rm GOE}(2^{2/3} s),
\end{equation}
where $F_{\rm GOE}$ is the GOE Tracy-Widom distribution function~\cite{TW96}.
\end{thm}
In~\cite{BKS12} the distribution of the position where the maximum of ${\cal A}_2(v)-v^2$ is attained has been derived. Due to the quadratic term it is localized and bounds can be found in~\cite{CH11,QR12b}. These bounds can be compared with our Lemma~\ref{LemmaBoundOnR}, where we obtain a Gaussian bound in $M$ of the probability that the maximizers is not in a main region of order $\Or(M N^{2/3})$ (uniformly for all $N$ large enough).

\begin{figure}
\begin{center}
\psfrag{n}[lc]{$n$}
\psfrag{m}[lc]{$m$}
\psfrag{L}[lc]{$\LL$}
\psfrag{Lf}[lc]{$\Lflat$}
\psfrag{pi}[lc]{$\pi$}
\psfrag{E}[lc]{$E_N(w)$}
\includegraphics[height=3.5cm]{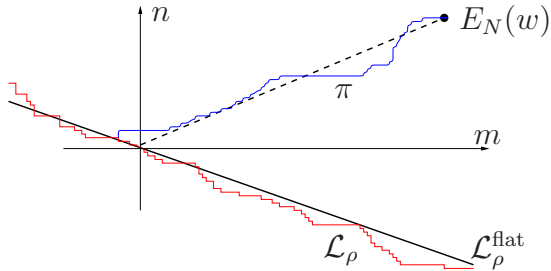}
\caption{The last passage percolation setting considered in Theorem~\ref{ThmMainLPP}. The maximizer $\pi$ from $\LL$ (red) to $E_N(w)$ starts in a $\Or(N^{2/3})$-neighborhood of the origin. The straight thick line represents $\Lflat$.}
\label{FigLPP}
\end{center}
\end{figure}

\begin{remark}
From the work on KPZ equation of Remenik and Quastel~\cite{QR16} it is conjectured that for KPZ growth models, if the initial configuration is flat with subdiffusive scaling, then the limiting distribution is the same as for the flat case (see Theorem~1.5 and subsequent remarks in~\cite{QR16}). In the LPP framework this corresponds to have $\LL$ replaced by a (possibly random) down-right line, which at distance $X$ from the origin has fluctuations at most $\Or(|X|^\delta)$ for some $\delta<1/2$. Theorem~\ref{ThmMainLPP} confirms it for general densities (since in that case ${\cal R}=0$); compare with~\cite{CFS16,CLW16} for the analogue result at $\rho=1/2$.
\end{remark}

The proof of the main theorem (Theorem~\ref{ThmMainLPP}) is in his core probabilistic and it is based on the comparison of the LPP problem from a horizontal line to $E_N(w)$, where the line is around the region where the LPP from $\LL$ to $E_N(w)$ is achieved. If we look the maximizers from the $E_N(w)$ position backwards, this is equivalent to consider the LPP from $(0,0)$ to a horizontal line crossing $(\gamma^2 n,n)$ for some $\gamma\in (0,\infty)$ with $n$ proportional to $N$. Therefore consider the following LPP setting: for $i,j\geq 1$, let $\omega_{i,j}$ be i.i.d.\ ${\rm Exp}(1)$ random variables, $\omega_{i,j}=0$ for $i\leq 0$ or $j\leq 0$.

The estimate from law of large numbers for the LPP from the origin to $(M,N)$ is given by $(\sqrt{M}+\sqrt{N})^2$ (as shown by Rost~\cite{R81} in the TASEP setting). Due to KPZ scaling we define the rescaled last passage time\footnote{Here and below we will not write the integer parts explicitly in the entries of the LPP.}
\begin{equation}\label{eq3.1}
L^{\rm resc,h}_n(u):=\frac{L_{(0,0)\to(\gamma^2 n+ \beta_1 u n^{2/3},n)}- n(1+\sqrt{\gamma^2+ \beta_1 u n^{-1/3}})^2}{\beta_2 n^{1/3}},
\end{equation}
where we set $\beta_1=2(1+\gamma)^{2/3}\gamma^{4/3}$ and $\beta_2=(1+\gamma)^{4/3} \gamma^{-1/3}$. The coefficient $\beta_2$ is chosen to have the one-point distribution given by the GUE Tracy-Widom distribution~\cite{TW94}, as shown by Johansson in Theorem~1.6 of~\cite{Jo00b}. The coefficient $\beta_1$ is chosen such that the limit process converges to the Airy$_2$ process~\cite{PS02}, ${\cal A}_2$. The finite-dimensional convergence to the Airy$_2$ process is a special case of~\cite{BF07,BP07,SI07}. Note that since
\begin{equation}
n(1+\sqrt{\gamma^2+ \beta_1 u n^{-1/3}})^2 = (1+\gamma)^2 n + 2 u (1+\gamma)^{5/3}\gamma^{1/3} n^{2/3}-\beta_2 u^2 n^{1/3}+\Or(1)
\end{equation}
we can replace in (\ref{eq3.1}) also the approximation of the LLN until the order $n^{1/3}$ only without any relevant changes.

\begin{thm}\label{ThmTightness}
Fix any $M\in (0,\infty)$. Then, $u\mapsto L^{\rm resc}_{n}(u)$ is tight in the space of continuous functions on $[-M,M]$, ${\cal C}([-M,M])$.
\end{thm}

As a direct consequence of the convergence of finite-dimensional distributions and tightness we have:
\begin{cor}\label{corWeakConvergence}
For any given finite $M>0$, $u\mapsto L_n^{\rm resc}(u)$ converges weakly to an Airy$_2$ process $u\mapsto {\cal A}_2(u)$ in ${\cal C}([-M,M])$.
\end{cor}

The next result which is in itself interesting is a bound of the exit point probability for the stationary situation, which can be achieved (see more details in Section~\ref{sectExitPoints}) if we consider the LPP as before but with extra random variables if $i=0$ or $j=0$, namely with
\begin{equation}\label{stationaryboundary}
 \omega_{i,j}= \begin{cases}
  0 & i=0, j=0,\\
  {\rm Exp}(1-\rho) & i\geq 1, j=0,\\
  {\rm Exp}(\rho) & i=0, j\geq 1,\\
  {\rm Exp}(1) & i\geq 1, j\geq 1.\\
 \end{cases}
\end{equation}
Here ${\rm Exp}(a)$ denotes exponential random variables with parameter $a$ (thus average $1/a$). For the LPP with boundary conditions \eqref{stationaryboundary} we define the \emph{exit point} as the last point of a path $\pi_{(0,0)\rightarrow(m,n)}$ on the $x$-axis or the $y$-axis. Since we need to distinguish whether the exit point is on the $x$- or on the $y$-axis, we introduce a random variable $Z^{\rho}(m,n)\in\mathbb{Z}$ such that, if $Z^{\rho}(m,n)>0$, then the exit point is $(Z^{\rho}(m,n),0)$, as if $Z^{\rho}(m,n)<0$, then the exit point is $(0,-Z^{\rho}(m,n))$.

\begin{lem}[Exit point probability]\label{lemmaExitPoint}
Let $\kappa>0$ be given and set
\begin{equation}
\rho_\pm=\rho_0\pm \kappa n^{-1/3}\textrm{ with }\rho_0=\frac{1}{\gamma+1}.
\end{equation}
Then there exists a $n_0$ such that for all $n\geq n_0$,
\begin{equation}
\begin{aligned}
\Pb(Z^{\rho_+}(\gamma^2 n,n)>0)&\geq 1-C \exp(-c \kappa^2),\\
\Pb(Z^{\rho_-}(\gamma^2 n,n)<0)&\geq 1-C \exp(-c \kappa^2),
\end{aligned}
\end{equation}
for some constants $C,c$ independent of $\kappa$ (and which can be taken uniform for $\gamma$ in a bounded set).
\end{lem}

A simple change of variables gives the following result.
\begin{cor}\label{corExitPoint}
In the settings of Lemma~\ref{lemmaExitPoint}, for any given $M>0$ and $\kappa$ satisfying
\begin{equation}
\tilde\kappa=\kappa - M \gamma^{1/3}(1+\gamma)^{-4/3}>0
\end{equation}
it holds
\begin{equation}
\begin{aligned}
\Pb(Z^{\rho_+}(\gamma^2n-\beta_1 M n^{2/3},n)\geq 0) & \geq 1-C\exp(-c \tilde \kappa^2),\\
\Pb(Z^{\rho_-}(\gamma^2n+\beta_1 M n^{2/3},n)\geq 0) & \geq 1-C\exp(-c \tilde \kappa^2).
\end{aligned}
\end{equation}
\end{cor}

\subsection{Universality for TASEP}
The LPP with $\Lflat$ as initial set corresponds to TASEP in continuous time with initial condition $x_k(0)=-\lfloor k/\rho\rfloor$, $k\in\Z$. We have the following universality result for the one-point fluctuations for TASEP with flat initial conditions for any density $\rho\in (0,1)$. For the more general initial condition, in terms of height Assumptions~A and~B rewrite as follows.

\noindent \textbf{Assumption~A:}
\begin{equation}
\lim_{L\to\infty} \frac{h(2v\chi^{1/3}L^{2/3},0)-2v (1-2\rho)\chi^{1/3} L^{2/3}}{2\chi^{2/3}L^{1/3}}={\cal R}(v)=\sqrt{2}\sigma{\cal B}(v),
\end{equation}
weakly on the space of continuous functions on bounded sets, where ${\cal B}$ is a two-sided Brownian motion and $\sigma\geq 0$ a coefficient. The stationary initial condition is $\sigma=1$, while the flat initial condition is $\sigma=0$.

\noindent \textbf{Assumption~B:} For any given $\delta>0$ and $M>0$, there exists a $L_0$ such that for all $L\geq L_0$,
\begin{equation}
\Pb(h(2v\chi^{1/3}L^{2/3},0)-2v (1-2\rho)\chi^{1/3} L^{2/3}\geq -\delta v^2 L^{1/3}\textrm{ for all }|v|\geq M)\geq 1-Q(M),
\end{equation}
with $Q$ independent on $L$ and $\lim_{M\to\infty} Q(M)=0$.

\begin{thm}\label{ThmMainTASEP}
Let $\rho\in (0,1)$ and set $\chi=\rho(1-\rho)$. Then, for any $s\in\R$,
\begin{equation}
\begin{aligned}
&\lim_{t\to\infty}\Pb\left(h((1-2\rho)t+2 w \chi^{1/3} t^{2/3},t)\geq (1-2\chi)t+2w (1-2\rho) \chi^{1/3} t^{2/3}-2 s \chi^{2/3} t^{1/3}\right)\\
&=\Pb\left(\max_{v\in\R} \{{\cal A}_2(v)-(v-w)^2+{\cal R}(v)\}\leq s\right).
\end{aligned}
\end{equation}
\end{thm}
\begin{proof}
The first equality follows from (\ref{1.11}). The rest is a direct consequence of Theorem~\ref{ThmMainLPP} and the relation (\ref{eq2.4}).
\end{proof}

The flat TASEP is the special case ${\cal R}=0$ and the result is independent of $w$ since the Airy$_2$ process is stationary. Thus we have proven the following result, which motivated the study of this paper.
\begin{cor}\label{CorMainTASEP}
Consider TASEP with flat initial condition and density $\rho\in (0,1)$, and set $\chi=\rho(1-\rho)$. Then, for any $s\in\R$,
\begin{equation}
\begin{aligned}
\lim_{t\to\infty}\Pb\left(J(t)\geq \chi t - s \chi^{2/3} t^{1/3}\right)
&=\lim_{t\to\infty}\Pb\left(h((1-2\rho)t,t)\geq (1-2\chi)t-2 s \chi^{2/3} t^{1/3}\right)\\
&=\Pb\left(\max_{v\in\R}\{ {\cal A}_2(v)-v^2\}\leq s\right) = F_{\rm GOE}(2^{2/3} s).
\end{aligned}
\end{equation}
\end{cor}

\section{Comparison with stationary LPP and proof of Theorem~\ref{ThmTightness}}\label{sect:tightness}
In this section we will prove tightness of the process $L^{\rm resc,h}_n$. This mainly follows the approach of Cator and Pimentel~\cite{CP15b}. The key observation in~\cite{CP15b} is that the increments of the LPP with end-points on a horizontal line can be bounded by the increments of the LPP for the stationary case on the set of events where the ``exit point'' is on the right or the left of the origin. Then the idea is to consider stationary LPP with slightly higher/lower density so that the given exit point events are highly probable and at the same time the increments of the LPP are controlled by the ones in the stationary LPPs. In~\cite{CP15b} the case of the Hammersley process was studied in details and it was stated the result for the exponential random variable along the diagonal only, i.e.\ $\gamma=1$. The proof of the latter is left to the reader as it was mentioned that it is similar to the case of the Hammersley.

We have a few reasons to present the details for the result with generic densities:\\[0.5em]
(a) here we consider the space of continuous functions instead of the c\`adl\`ag functions and there are some minor twists which have to be taken into account for generic density $\rho\neq 1/2$;\\[0.5em]
(b) we get a much stronger bound for the exit point distributions with respect to~\cite{CP15b} (see Lemma~\ref{lemmaExitPoint});\\[0.5em]
(c) we derive an estimate on the increments, which is not needed for proving tightness, but it is the key for the control of the probability that the maximizer of the LPP from $\LL$ to $E_N(w)$ is localized: the derivation of this result is noticeably simplified with respect to the previous papers~\cite{CLW16} (they made use of a Brownian-Gibbs property) and~\cite{CFS16} (an ad-hoc comparison with half-line problem with slope $-1$ was used).

\subsection{Stationary LPP and exit points}\label{sectExitPoints}
Let us now explain what we mean with stationary LPP with density $\rho\in(0,1)$ and report a result of Bal{\'a}zs, Cator and Sepp{\"a}l{\"a}inen~\cite{BCS06}.
Consider the LPP as given by (\ref{stationaryboundary}). We denote by $L^\rho(m,n)$ the last passage percolation from $(0,0)$ to $(m,n)$ in this setting, while we use $L(m,n)$ for the last passage percolation from $(0,0)$ to $(m,n)$ if we set $\omega_{i,0}=\omega_{0,j}=0$.

The boundary conditions \eqref{stationaryboundary} correspond to a TASEP starting from the stationary Bernoulli($\rho$) measure, conditioned on $\eta_0(0)=0$ and $\eta_1(0)=1$. Let $P_0(t)$ be the position at time $t$ of the particle which started in $1$ at time $0$, and $H_0(t)$ be the position at time $t$ of the hole which started in $0$ at time $0$. It was shown in Corollary~3.2 of~\cite{BCS06} (as a corollary of Burke's theorem~\cite{Bur56}) that $P_0(t)-1$ and $-H_0(t)$ are two independent Poisson processes with jump rates $1-\rho$ and $\rho$. They extended the result to get independent increments also in the bulk of the system. The result we will use is the following:
\begin{lem}[Special case of Lemma~4.2 of~\cite{BCS06}]\label{LemmaStatIncrements}
Fix any $n\geq 1$. Then the increments
\begin{equation}
\{L^{\rho}(m+1,n)-L^{\rho}(m,n),m\geq 1\}
\end{equation}
are are i.i.d.\ exponential random variables with parameter $1-\rho$.
\end{lem}

With this definition we have the following lower and upper bounds in the increments of the process $m\mapsto L(m,n)$ that we want to study:
\begin{lem}[Lemma~1 of~\cite{CP15b}] \label{LemmaIncrementsBounds}
 Let $0\le m_1\le m_2$. Then if $Z^{\rho}(m_1,n)\geq 0$, it holds
 \begin{equation}
  L(m_2,n)-L(m_1,n)\le L^{\rho}(m_2,n)-L^{\rho}(m_1,n),
  \label{upperbound}
 \end{equation}
while, if $Z^{\rho}(m_2,n)\leq 0$, then we have
 \begin{equation}
  L(m_2,n)-L(m_1,n)\ge L^{\rho}(m_2,n)-L^{\rho}(m_1,n).
  \label{lowerbound}
 \end{equation}
\end{lem}

From the law of large numbers results one easily obtains that $Z^\rho(\gamma^2n,n)$ is typically around $0$ (it will fluctuates over a $n^{2/3}$ scale), if one chooses $\rho=1/(\gamma+1)$. Therefore we set
\begin{equation}
\rho_\pm=\rho_0\pm \kappa n^{-1/3}\textrm{ with }\rho_0=\frac{1}{\gamma+1}.
\end{equation}
The choice of $n^{-1/3}$ is due to the fact that the increments of the scaled process are just increased/decreased by a finite amount (proportional to $\kappa$), but on the other hand $\Pb(Z^{\rho_+}(\gamma^2n,n)>0)$ and $\Pb(Z^{\rho_-}(\gamma^2n,n)<0)$ goes to $1$ as $\kappa\to\infty$. The first step is to get an estimate on these probabilities.

\subsection{Bounds on exit points}
Now we want to derive a bound on $\Pb(Z^{\rho_+}(\gamma^2n,n)>0)$ and on $\Pb(Z^{\rho_-}(\gamma^2n,n)<0)$. The last passage time $L^\rho$ is the maximum between the last passage time from $(0,1)$ and the one from $(1,0)$, since any up-right path from $(0,0)$ has to go through one of these points. These LPP are denoted by
\begin{equation}
L^\rho_{-}(m,n)=L_{(0,0)\to(1,0)\to(m,n)},\quad L^\rho_{|}(m,n)=L_{(0,0)\to (0,1)\to(m,n)}.
\end{equation}
In terms of these two random variables, we have
\begin{equation}
\begin{aligned}
\Pb(Z^{\rho_+}(\gamma^2n,n)>0)&=\Pb\left(L^{\rho_+}_{-}(\gamma^2 n,n)>L^{\rho_+}_{|}(\gamma^2 n,n)\right),\\
\Pb(Z^{\rho_-}(\gamma^2n,n)<0)&=\Pb\left(L^{\rho_-}_{|}(\gamma^2 n,n)>L^{\rho_-}_{-}(\gamma^2 n,n)\right).
\end{aligned}
\end{equation}

Now we are ready to prove Lemma~\ref{lemmaExitPoint} and Corollary~\ref{corExitPoint}.

\begin{proof}[Proof of Lemma~\ref{lemmaExitPoint}]
By symmetry of the problem under the exchanges $\gamma \to 1/\gamma$ and $\rho\to 1-\rho$ it is enough to deal with the first estimate. We are going to prove that $\Pb(Z^{\rho_+}(\gamma^2n,n)<0)\leq C \exp(-c \kappa^2)$.

First notice that for any $x\in\R$ we have
\begin{equation}\label{eq3.11}
\begin{aligned}
\Pb(Z^{\rho_+}(\gamma^2n,n)<0) &= \Pb\left(L^{\rho_+}_{-}(\gamma^2 n,n)<L^{\rho_+}_{|}(\gamma^2 n,n)\right)\\
&\leq \Pb\left(L^{\rho_+}_{-}(\gamma^2 n,n)\leq x\right) + \Pb\left(L^{\rho_+}_{|}(\gamma^2 n,n)>x\right).
\end{aligned}
\end{equation}
Further, since for $\kappa>0$ we have $\rho_+>\rho_0$, and thus $\E(\omega_{0,i})=1/\rho_+<1/\rho_0$, implying
\begin{equation}\label{eq3.12}
\Pb\left(L^{\rho_+}_{|}(\gamma^2 n,n)>x\right)\leq \Pb\left(L^{\rho_0}_{|}(\gamma^2 n,n)>x\right).
\end{equation}
The bounds of Lemma~\ref{LemmaBoundDistributions} below with $x=(1+\gamma)^2 n +a \kappa^2 \beta_2 n^{1/3}$ (where we can choose any value $a\in (0,(1+\gamma)^{8/3}\gamma^{-2/3})$) together with (\ref{eq3.11}) and (\ref{eq3.12}) give the desired result.
\end{proof}

\begin{proof}[Proof of Corollary~\ref{corExitPoint}]
Setting $\tilde\gamma^2 n = \gamma^2 n \pm \beta_1 M n^{2/3}$ and $\frac{1}{1+\gamma}\pm\kappa n^{-1/3}=\frac{1}{1+\tilde\gamma}\pm\tilde \kappa n^{-1/3}$ we find the value of $\tilde\kappa$. Then the bound follows by Lemma~\ref{lemmaExitPoint}.
\end{proof}

\begin{lem}\label{LemmaBoundDistributions}
Let $x=(1+\gamma)^2 n +a \kappa^2 \beta_2 n^{1/3}$ with $a\in (0,(1+\gamma)^{8/3}\gamma^{-2/3})$. Then, uniformly for $n$ large enough, we have
\begin{equation}
\begin{aligned}
\Pb\left(L^{\rho_0}_{|}(\gamma^2 n,n)>x\right)&\leq C e^{-c \kappa^2},\\
\Pb\left(L^{\rho_+}_{-}(\gamma^2 n,n)\leq x\right) &\leq C e^{-c \kappa^3},
\end{aligned}
\end{equation}
for some $\kappa$-independent constants $C,c\in (0,\infty)$ ($c$ is depending on $a$).
\end{lem}
\begin{proof}
Denoting $L^{\rho_0,\rm resc}:=\frac{L^{\rho_0}_{|}(\gamma^2 n,n)-(1+\gamma)^2 n}{\beta_2 n^{1/3}}$, the first inequality becomes an estimate on $1-\Pb(L^{\rho_0,\rm resc}\leq a\kappa^2)$. The distribution of $L^{\rho_0,\rm resc}$ has been studied in~\cite{BBP06} in the framework of sample covariance matrices. One can use the connection of this LPP to a rank-one problem in sample covariance matrices (see Section 6 of~\cite{BBP06}) to recover the result. Let us explain how it goes.

From (62) of~\cite{BBP06} we have that
\begin{equation}
\Pb(L^{\rho_0,\rm resc}\leq \xi)=\det\left(\Id-K_n\right)_{L^2(\R_+)}
\end{equation}
where $K_n$ is a trace-class operator acting on $L^2(\R_+)$. The integral kernel of $K_n$ can be expressed as
\begin{equation}
K_n(u,v)=\int_{\R_+}d\lambda H_n(u,\lambda)J_n(\lambda,v),
\end{equation}
where $H_n(u,v)={\cal H}(\xi+u+v)$ and $J_n(u,v)={\cal J}(\xi+u+v)$ with $\mathcal{H}$, $\mathcal{J}$ given in (93)-(96) of~\cite{BBP06}. Using the triangular inequality and a standard inequality on Fredholm determinants (see e.g.\ Theorem~3.4 of~\cite{Sim00}) we have
\begin{equation}\label{eq3.15}
\begin{aligned}
|1-\det(\Id-K_n)|&\leq |1-\det(1-K_\infty)| + |\det(1-K_\infty)-\det(1-K_n)| \\
&\leq (\|K_\infty\|_1+\|K_\infty-K_n\|_1) \exp(\|K_\infty\|_1+\|K_n\|_1+1).
\end{aligned}
\end{equation}
The limits of $\cal H$ and $\cal J$ are denoted by ${\cal H}_\infty$ and ${\cal J}_\infty$ and they are given in (120) and (122) of~\cite{BBP06}. For $k=1$ ${\cal H}_\infty(u)=e^{-\e u} \int_{\R_+} \Ai(\xi+\lambda+u)d\lambda$ and ${\cal J}_\infty(u)=e^{\e u} \Ai'(\xi+u)$ with $\e>0$ being any small constant. Using triangular inequalities and the identity $\|A B\|_1\leq \|A\|_{\rm HS} \|B\|_{\rm HS}$ (see e.g.\ Theorem VI.22 of~\cite{RS78I}) we can bound each of the norms in (\ref{eq3.15}) by a finite sum of product of two of the following Hilbert-Schmidt norms,
\begin{equation}
\|H_\infty\|_{\rm HS},\quad  \|J_\infty\|_{\rm HS}, \quad \|H_\infty-H_n\|_{\rm HS},\quad \|J_\infty-J_n\|_{\rm HS},
\end{equation}
As a function of $\xi$, the latter two have exponential bounds (see Proposition~3.1 of~\cite{BBP06}) uniformly for $n$ large enough, while the first two have (super-)exponential decay from the known asymptotics of the Airy functions (e.g., $|\Ai(x)|\leq e^{-x}$ and $|\Ai'(x)|\leq e^{-x}$, for all $x\in\R$).

To prove the second inequality, it is enough to have a bound on the probability for a lower bound for $L^{\rho_+}_{-}$. For any choice of $\xi_0>0$, we have
\begin{equation}
\begin{aligned}
L^{\rho_+}_{-}(\gamma^2 n,n) &\geq L^{\rho_+}(\xi_0 n^{2/3},0)+L^{\rho_+}_{(\xi_0 n^{2/3},0)\to (\gamma^2 n,n)}\\
&\geq L^{\rho_+}(\xi_0 n^{2/3},0)+L_{(\xi_0 n^{2/3},0)\to (\gamma^2 n,n)},
\end{aligned}
\end{equation}
where the $L$ without $\rho_+$ means the LPP with all $\omega$'s to be ${\rm Exp}(1)$. Then
\begin{equation}\label{eq3.19}
\Pb\left(L^{\rho_+}_{-}(\gamma^2 n,n)\leq x\right) \leq \Pb\left(L^{\rho_+}(\xi_0 n^{2/3},0)+L_{(\xi_0 n^{2/3},0)\to (\gamma^2 n,n)}\leq x\right).
\end{equation}
Let us see what is a good choice for $\xi_0$. The estimate from the law of large numbers gives
\begin{equation}\label{eq3.22}
L^{\rho_+}(\xi_0 n^{2/3},0) \simeq \xi_0 n^{2/3}/(1-\rho_+)=\tfrac{1+\gamma}{\gamma}\xi_0 n^{2/3}+\tfrac{(1+\gamma)^2}{\gamma^2} \xi_0\kappa n^{1/3}+O(1)
\end{equation}
and
\begin{equation}\label{eq3.23}
L_{(\xi_0 n^{2/3},0)\to (\gamma^2 n,n)} \simeq \left(\sqrt{n}+\sqrt{\gamma^2n-\xi_0 n^{2/3}}\right)^2=(1+\gamma)^2 n-\tfrac{1+\gamma}{\gamma}\xi_0 n^{2/3}-\frac{\xi_0^2}{4\gamma^3}n^{1/3}+O(1).
\end{equation}
The sum of (\ref{eq3.22}) and (\ref{eq3.23}) (up to $O(n^{1/3})$) is maximal for $\xi_0=2\gamma(1+\gamma)^2\kappa$, which is the value that we choose.
Let us define the rescaled LPP by
\begin{equation}
\begin{aligned}
L^{\rm resc}_{-}&=\frac{L^{\rho_+}(\xi_0 n^{2/3},0)-\left(\frac{1+\gamma}{\gamma}\xi_0 n^{2/3}+\frac{(1+\gamma)^2}{\gamma^2} \xi_0\kappa n^{1/3}\right)}{n^{1/3}},\\
L^{\rm resc}_{\rm bulk}&=\frac{L_{(\xi_0 n^{2/3},0)\to (\gamma^2 n,n)}-\left((1+\gamma)^2 n-\frac{1+\gamma}{\gamma}\xi_0 n^{2/3}-\frac{\xi_0^2}{4\gamma^3}n^{1/3}\right)}{n^{1/3}}
\end{aligned}
\end{equation}

Since $x=(1+\gamma)^2 n +a \kappa^2 \beta_2 n^{1/3}$, we have that
\begin{equation}
(\ref{eq3.19})\leq \Pb\left(L^{\rm resc}_{-}+L^{\rm resc}_{\rm bulk}\leq -\tilde s\right)\leq \Pb\left(L^{\rm resc}_{-}\leq -\tilde s/2\right)+\Pb\left(L^{\rm resc}_{\rm bulk}\leq -\tilde s/2\right)
\end{equation}
with $\tilde s=\left((1+\gamma)^4/\gamma-a\beta_2\right)\kappa^2$.

For any $a\in (0,(1+\gamma)^{8/3}\gamma^{-2/3})$ we have $\tilde s>0$.  Then, uniformly for $n$ large enough, by Proposition~\ref{PropBounds}(c) we have\footnote{The constant $c$ is not the same as in Proposition~\ref{PropBounds}(c), due to the $1/2$ term and the fact that $L^{\rm resc}_{\rm bulk}$ converges to a GUE Tracy-Widom distribution once divided by $\beta_2$.}
\begin{equation}
\Pb\left(L^{\rm resc}_{\rm bulk}\leq -\tilde s/2\right)\leq C e^{-c \tilde s^{3/2}}= C e^{-\tilde c \kappa^3}
\end{equation}
for some constants $C,c,\tilde c\in (0,\infty)$.

To bound the distribution of $L^{\rm resc}_{-}$, note that $L^{\rho_+}(\xi_0 n^{2/3},0)$ is a sum of $\lfloor\xi_0 n^{2/3}\rfloor$ i.i.d.\ random variables ${\rm Exp}(1-\rho_+)$. Let $X_i$ i.i.d.\ ${\rm Exp}(1-\rho_+)$ random variables. Consider the centered random variables $Y_i=1/(1-\rho_+)-X_i$. Set $\hat s=\tilde s n^{1/3}/2$ and $N=\lfloor \xi_0 n^{2/3}\rfloor$. Then by the exponential Tchebishev inequality,
\begin{equation}\label{eq3.25}
\Pb\left(L^{\rm resc}_{-}\leq -\tilde s/2\right)=\Pb\bigg(\sum_{i=1}^{N} Y_i\geq \hat s\bigg)\leq\inf_{t\ge0} e^{-\hat s t}\left(\mathbb{E}\left(e^{t Y_1}\right)\right)^N.
\end{equation}
We have $\mathbb{E}\left(e^{t Y_1}\right)=e^{t/(1-\rho_+)}/(1+t/(1-\rho_+))$ and thus $(\ref{eq3.25})\leq \exp(\inf_{t\geq 0} I(t))$ with $I(t)=N t/(1-\rho_+)+N \ln((1-\rho_+)/(t+1-\rho_+))-\hat s t$. A simple computation gives
\begin{equation}\label{3.26}
\begin{aligned}
\inf_{t\geq 0} I(t)&=\hat s (1-\rho_+)+N \ln(1-\hat s (1-\rho_+)/N)\\
&=-\frac{\tilde s^2\gamma^2}{8\xi_0(1+\gamma)^2}+\Or(n^{-1/3})\leq -\hat c \kappa^3,
\end{aligned}
\end{equation}
for some constant $\hat c$ (which can be taken independent on $n\geq n_0$, $n_0$ large enough), since $\xi_0\sim\kappa$ and $\tilde s\sim\kappa^2$ as well.
\end{proof}

\subsection{Tightness}
Now we prove tightness of the rescaled process $L^{\rm resc,h}_n$ (see (\ref{eq3.1})). Following the ideas in~\cite{CP15b} we prove it using the bounds of Lemma~\ref{LemmaIncrementsBounds} together with the estimates of Lemma~\ref{lemmaExitPoint} and of the fluctuations of sums of i.i.d.\ random variables.

First let us see what Lemma~\ref{LemmaIncrementsBounds} becomes for the rescaled processes. This bounds will be used to show tightness, but also to control the fluctuations beyond the central region of the maximisation problem (see Lemma~\ref{LemmaBoundOnR}). Let us shortly recall the scaling (\ref{eq3.1}) under which $L^{\rm resc,h}_n$ converges in the sense of \emph{finite-dimensional distributions}~\cite{BF07,BP07,SI07} to the Airy$_2$ process, ${\cal A}_2$,
\begin{equation}\label{eq3.1b}
L^{\rm resc,h}_n(u):=\frac{L_{(0,0)\to(\gamma^2 n+ \beta_1 u n^{2/3},n)}- \left((1+\gamma)^2 n + 2 u (1+\gamma)^{5/3}\gamma^{1/3} n^{2/3}-\beta_2 u^2 n^{1/3}\right)}{\beta_2 n^{1/3}},
\end{equation}
with $\beta_1=2(1+\gamma)^{2/3}\gamma^{4/3}$ and $\beta_2=(1+\gamma)^{4/3} \gamma^{-1/3}$.

\begin{lem}\label{lemmaBoundsRescaledProcesses}
Let us define
\begin{equation}\label{eq3.30}
B^{\rho_{\pm}}_{n}(u):=\frac{L^{\rho_{\pm}}(\gamma^2n+\beta_1 u n^{2/3},n)-(L^{\rho_{\pm}}(\gamma^2 n,n)+\frac{1}{1-\rho_\pm}\beta_1 u n^{2/3})}{\beta_2 n^{1/3}}.
\end{equation}
For any fixed constants $M_1,M_2$, consider any two points satisfying $-M_1\leq v\leq u\leq M_2$. Then we have:\\
(a) If $Z^{\rho_+}(\gamma^2n-\beta_1 M_1 n^{2/3},n)\geq 0$, then
\begin{equation}\label{upperbound_modulus}
L^{\rm resc,h}_n(u)-L^{\rm resc,h}_n(v)\leq B^{\rho_+}_n(u)-B^{\rho_+}_n(v) + (u^2-v^2)+2\beta_2 \kappa (u-v)+\Or(n^{-1/3}).
 \end{equation}
(b) If $Z^{\rho_{-}}(\gamma^2n+\beta_1 M_2 n^{2/3},n)\leq 0$, then
\begin{equation}  \label{lowerbound_modulus}
L^{\rm resc,h}_n(u)-L^{\rm resc,h}_n(v)\geq B^{\rho_-}_n(u)-B^{\rho_-}_n(v) +(u^2-v^2)-2\beta_2 \kappa (u-v)+\Or(n^{-1/3}).
 \end{equation}
Here $\Or(n^{-1/3})$ is uniformly for $\kappa$ and $\gamma$ in bounded sets of $(0,\infty)$.
\end{lem}
\begin{proof}
We wrote the conditions on the left-most and right-most point, since by monotonicity they imply the conditions needed to apply Lemma~\ref{LemmaIncrementsBounds} for the full interval $[-M_1,M_2]$.
By Lemma~\ref{LemmaIncrementsBounds} and the definition of the scalings (\ref{eq3.1b}) and (\ref{eq3.30}) we have
\begin{equation}\label{eq3.32}
\begin{aligned}
L^{\rm resc,h}_n(u)-L^{\rm resc,h}_n(v)&\leq B^{\rho_+}_n(u)-B^{\rho_+}_n(v) + (u^2-v^2)\\
&+ \left(\frac{\beta_1}{1-\rho_+}-2(1+\gamma)^{5/3}\gamma^{1/3} \right)\frac{(u-v)}{\beta_2} n^{1/3}.
\end{aligned}
\end{equation}
Using the explicit expressions for $\beta_1$, $\beta_2$, and $\rho_+$ we get (\ref{upperbound_modulus}).

Similarly, we have
\begin{equation}\label{eq3.32b}
\begin{aligned}
L^{\rm resc,h}_n(u)-L^{\rm resc,h}_n(v)&\geq B^{\rho_-}_n(u)-B^{\rho_-}_n(v) + (u^2-v^2)\\
&+ \left(\frac{\beta_1}{1-\rho_-}-2(1+\gamma)^{5/3}\gamma^{1/3} \right)\frac{(u-v)}{\beta_2} n^{1/3},
\end{aligned}
\end{equation}
giving (\ref{lowerbound_modulus}).
\end{proof}

Let us denote the modulus of continuity for the rescaled process $L^{\rm resc,h}_n$ in the interval $[-M,M]$ by $\varpi_n(\delta)$:
\begin{equation}\label{modulus}
\varpi_n(\delta)=\sup_{\begin{smallmatrix}|u|,|v|\leq M\\|u-v|\leq\delta\end{smallmatrix}}|L^{\rm resc,h}_n(u)-L^{\rm resc,h}_n(v)|.
\end{equation}

\begin{proof}[Proof of Theorem~\ref{ThmTightness}]
First of all, notice that the random variable $L^{\rm resc,h}_n(0)$ is tight, see the upper and lower tail estimates in Proposition~\ref{PropBounds}. Thus to show tightness it remains to control the modulus of continuity, namely we need to prove that for any $\e,\tilde\e>0$, there exists a $\delta>0$ and a $n_0$ such that
\begin{equation}\label{eq3.34}
\Pb(\varpi_n(\delta)\geq \e)\leq \tilde\e,
\end{equation}
for all $n\geq n_0$.

For any $\e>0$, for $n$ large enough, by Lemma~\ref{lemmaExitPoint} it holds
\begin{equation}
\Pb(\varpi_n(\delta)\geq \e)\leq 2 C e^{-c\kappa^2} + \Pb(\{\varpi_n(\delta)\geq \e\}\cap \{Z^{\rho_+}_M>0\}\cap\{Z^{\rho_-}_M<0\}),
\end{equation}
where we shorten $Z^{\rho_+}_M=Z^{\rho_+}(\gamma^2n-\beta_1 M n^{2/3},n)$ and $Z^{\rho_-}_M=Z^{\rho_-}(\gamma^2n+\beta_1 M n^{2/3},n)$.
From Lemma~\ref{lemmaBoundsRescaledProcesses}, for $|u|,|v|\leq M$ and $|u-v|\leq \delta$, if we choose $n$ large enough so that the $\Or(n^{-1/3})$ are smaller than $\delta$, then on the set $\{Z^{\rho_+}_M>0\}\cap\{Z^{\rho_-}_M<0\}$ we have
\begin{equation}
|L_n^{\rm resc}(u)-L_n^{\rm resc}(v)|\leq |B^{\rho_+}_n(u)-B^{\rho_+}_n(v)|+|B^{\rho_-}_n(u)-B^{\rho_-}_n(v)|+K(\delta,M,\kappa)
\end{equation}
with $K(\delta,M,\kappa)=(2M+1+2\beta_2\kappa)\delta$. Now choose $\delta$ small enough so that $K(\delta,M,\kappa)<\e/2$. Then, for all $n$ large enough,
\begin{equation}\label{eq3.36}
\begin{aligned}
&\Pb(\{\varpi_n(\delta)\geq \e\}\cap \{Z^{\rho_+}_M>0\}\cap\{Z^{\rho_-}_M<0\}) \\&\leq
\Pb\bigg(\sup_{\begin{smallmatrix}|u|,|v|\leq M\\|u-v|\leq\delta\end{smallmatrix}}|B^{\rho_+}_n(u)-B^{\rho_+}_n(v)|\geq \e/4\bigg)\\
&+\Pb\bigg(\sup_{\begin{smallmatrix}|u|,|v|\leq M\\|u-v|\leq\delta\end{smallmatrix}}|B^{\rho_-}_n(u)-B^{\rho_-}_n(v)|\geq \e/4\bigg).
\end{aligned}
\end{equation}
Dividing the interval $[-M,M]$ into pieces of length $\delta$ and using stationarity of the increments of $B^{\rho_\pm}$ (and $B^{\rho_\pm}(0)=0$) we readily have
\begin{equation}\label{eq3.37}
\Pb\bigg(\sup_{\begin{smallmatrix}|u|,|v|\leq M\\|u-v|\leq\delta\end{smallmatrix}}|B^{\rho_\pm}_n(u)-B^{\rho_\pm}_n(v)|\geq \e/4\bigg)\\
\leq\frac{2M}{\delta}\Pb\Big(\sup_{0\leq u\leq \delta}|B^{\rho_\pm}_n(u)|\geq \e/12\Big),
\end{equation}
compare e.g.\ with sentence around (5.60) in~\cite{Jo03b}. A short computation and the use of Donsker's invariance principle theorem imply that the processes $u\mapsto B^{\rho_\pm}_n(u)$ converges weakly in ${\cal C}([-M,M])$ to $u\mapsto \sigma {\cal B}(u)$, where ${\cal B}$ is a standard Brownian motion and $\sigma=\sigma(\gamma)=\sqrt{2\gamma/(1+\gamma)}$. This implies that for $n$ large enough,
\begin{equation}
{\rm r.h.s.~of~}(\ref{eq3.37})\leq \frac{8M}{\delta} \Pb\Big(\sup_{0\leq u\leq \delta}|{\cal B}(u)|\geq \e/12\Big)\leq \frac{8 M}{\delta} \exp\left(-\frac{\e^2}{288\,\delta\sigma^2}\right),
\end{equation}
where we use the bound $\Pb\left(\sup_{t\in[0,T]}|{\cal B}(t)|>\lambda\right)\leq e^{-\lambda^2/2T}$.

To resume, we have obtained that for any $\e>0$ and $n$ large enough, it holds for $\tilde\kappa=\kappa - M \gamma^{1/3}(1+\gamma)^{-4/3}>0$,
\begin{equation}
\Pb(\varpi_n(\delta)\geq \e)\leq 2 C e^{-c\tilde \kappa^2} +  \frac{8 M}{\delta}\exp\left(-\frac{\e^2}{288\,\delta\sigma^2}\right).
\end{equation}
For any fixed $\tilde\e>0$, we choose $\kappa$ large enough such that $2 C e^{-c\tilde \kappa^2}\leq\tilde\e/2$ and then $\delta$ small enough such that $\frac{8 M}{\delta}\exp(-\e^2/(288\,\delta\sigma^2))\leq \tilde\e/2$ for any $n$ large enough. This proves (\ref{eq3.34}).
\end{proof}

\section{Proof of Theorem~\ref{ThmMainLPP}}\label{sect:mainThmproof}
In this section we prove the main theorem of LPP. The proof consists in showing that the LPP converges to a variational process. One essentially shows that (a) the LPP from $\LL$ to $E_N(w)$ is with high probability the same as the LPP from a subset of $\LL$ of size $\Or(M N^{2/3})$, and (b) that in that region the LPP converges to the variational process of the theorem restricted to $|u|\leq M$. The most important novelty of our proof, with respect to the works in~\cite{CLW16,CFS16}, is part (a). In~\cite{CLW16} they first needed to prove a Brownian-Gibbs property for an associated non-intersecting line ensemble. In~\cite{CFS16} one bounded a Fredholm determinant of a half-line problem corresponding to density $\rho=1/2$ for TASEP (and this approach can not be extended to the generic $\rho\in (0,1)$ case).

\begin{proof}[Proof of Theorem~\ref{ThmMainLPP}]
Let us recall that we study the LPP from $\LL$ and $\Lflat$ to $E_N(w)$. From the law of large numbers of the point-to-point LPP, see Proposition~\ref{PropBounds}(a), by optimizing over the positions on $\Lflat$ we obtain that the maximizer starts around $0$ (in a $\Or(N^{2/3})$ neighborhood). Remember the definition of the points $A^{\rm flat}(v)$ and $A(v)$ given in (\ref{eqDefSflat}) and (\ref{eqDefS}).
For a fixed $M>0$, define the following LPP problems:
\begin{equation}
L_{M}=\max_{|v|\leq M} L_{A(v)\to E_N(w)}\quad\textrm{and}\quad L_{M^c}=\max_{|v|> M} L_{A(v)\to E_N(w)}.
\end{equation}
According to (\ref{eqMainThm}) we need to determine the $N\to\infty$ limit of
\begin{equation}\label{eq2.15}
\Pb\left(\max\{L_{M},L_{M^c}\}\leq S(s)\right),\quad S(s)=N/\chi + s \chi^{-2/3} N^{1/3}.
\end{equation}
For large $M$ (as we will show) one expects that $L_{M}>L_{M^c}$ with high probability. Thus we define the events
\begin{equation}
R_{M}=\{L_{M^c}> S(s)\},\quad G_{M}=\{L_{M}\leq S(s)\}.
\end{equation}
With these definitions we have
\begin{equation}\label{eq2.17}
(\ref{eq2.15})=\Pb\left(R_{M}^c\cap G_{M}\right)= \Pb\left(G_{M}\right)-\Pb\left(R_{M}\cap G_{M}\right).
\end{equation}
In Lemma~\ref{LemmaBoundOnR} we show that, $\Pb(R_{M}\cap G_{M})\leq C e^{-c M^{2}}+Q(M)$ uniformly in $N$, where the function $Q$ is the one in Assumption~B. This implies that
\begin{equation}
\lim_{M\to\infty}\lim_{N\to\infty} \Pb\left(R_{M}\cap G_{M}\right) = 0.
\end{equation}
Thus it remains to determine $\lim_{M\to\infty}\lim_{N\to\infty} \Pb(G_M)$.

The limit is obtained by first considering the last passage percolation problem from points on the horizontal line crossing $(0,0)$, see Figure~\ref{FigLPPzoom},
\begin{figure}
\begin{center}
\psfrag{A0}[cc]{$(0,0)$}
\psfrag{Av}[rc]{$A(v)$}
\psfrag{Avf}[cc]{$A^{\rm flat}(v)$}
\psfrag{ApMf}[cc]{$A^{\rm flat}(M)$}
\psfrag{AmMf}[cc]{$A^{\rm flat}(-M)$}
\psfrag{Avt}[cc]{$\widetilde A(v)$}
\psfrag{Apv}[rc]{$\widetilde A^+(v)$}
\psfrag{Amv}[lc]{$\widetilde A^-(v)$}
\psfrag{L}[lc]{$\LL$}
\psfrag{Lf}[lc]{$\Lflat$}
\includegraphics[height=4cm]{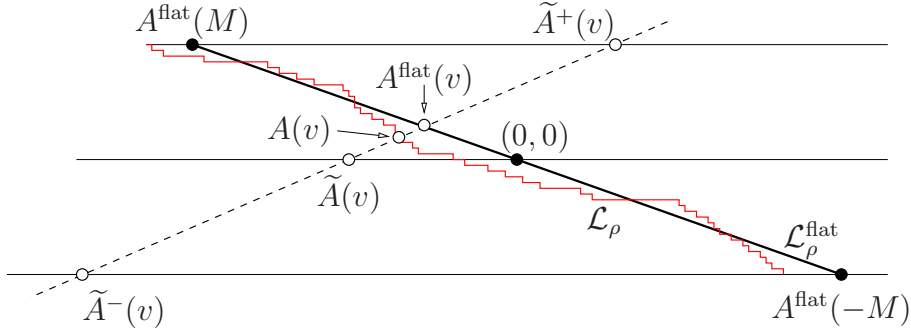}
\caption{Zoom of the LPP around the line relevant region of $\LL$ (red line) where the maximizers starts. For a given $v$, $\widetilde A^\pm(v)$, $\widetilde A(v)$, and $A(v)$ are on the same line, the line parallel to $\overline{(0,0),E_N(w)}$.}
\label{FigLPPzoom}
\end{center}
\end{figure}
for which the finite-dimensional distribution is known, and then using the functional slow-decorrelation result of Theorem~\ref{ThmFctSlowDec} we transport the fluctuations to the line $\LL$. We define
\begin{equation}\label{eqAtilde}
\widetilde A(v)=(-\alpha_1 v N^{2/3},0),\quad \alpha_1 = 2\frac{(1-\rho)^{2/3}}{\rho^{4/3}},
\end{equation}
and
\begin{equation}
\widetilde G_{M}=\Big\{\max_{|v|\leq M} L_{\widetilde A(v)\to E_N(w)}- \alpha_2 v N^{2/3}\leq S(s)\Big\},\quad \alpha_2=\frac{2}{\rho^{4/3}(1-\rho)^{1/3}}.
\end{equation}
In~\cite{BP07} it is shown\footnote{ The convergence of finite dimensional distributions can be also obtained from the finite-dimensional distributions along other lines using slow-decorrelation~\cite{CFP10b,Fer07}. For instance it can be obtained starting from the analogue result for the joint distributions of TASEP particle positions~\cite{BF07}; see~\cite{BFP09} for an application of this technique.} the \emph{convergence of finite dimensional distributions} of the rescaled process:
\begin{equation}\label{eq4.10}
\widetilde L^{\rm resc}_N(v):=\frac{L_{\widetilde A(v)\to E_N(w)}-(N/\chi +\alpha_2 v N^{2/3})}{\chi^{-2/3} N^{1/3}} \to {\cal A}_2(v)-(v-w)^2
\end{equation}
as $N\to\infty$, with ${\cal A}_2$ an Airy$_2$ process. In Theorem~\ref{ThmTightness} we show that as a process $v\mapsto \widetilde L^{\rm resc}_N(v)$ is \emph{tight} in the set of continuous functions with supremum norm, ${\cal C}([-M,M])$, extending the sense of convergence to the weak*-convergence.

The rescaled process we want to study is
\begin{equation}\label{defRescProcess}
L_N^{\rm resc}(v):=\frac{L_{A(v)\to E_N(w)}-N/\chi}{\chi^{-2/3} N^{1/3}}.
\end{equation}
In terms of the rescaled process, we indeed have
\begin{equation}
\Pb(G_M)=\Pb\Big(\max_{|v|\leq M} L_N^{\rm resc}(v)\leq s\Big).
\end{equation}
For any realization of initial condition, the random line $\LL$ passes in a neighborhood of the origin. Restricted to a $M N^{2/3}$-neighborood of the origin, by Assumption A we have that the points on $\LL$ are given by
\begin{equation}
A(v)=A^{\rm flat}(v)+\lambda(v) \mathbf{e}_\rho,\textrm{ with }\lambda(v)\simeq \chi^{-2/3} N^{1/3} {\cal R}(v)
\end{equation}
as $N\to\infty$. Define the set
\begin{equation}
F_\e=\Big\{\max_{|v|\leq M}|L_N^{\rm resc}(v)-\widetilde L_N^{\rm resc}(v)|\leq \e\Big\}.
\end{equation}
By Theorem~\ref{ThmFctSlowDec}, for any $\e>0$, $\lim_{N\to\infty} \Pb(F_\e)=1$. Thus, for any $\e>0$,
\begin{equation}
\lim_{M\to\infty}\lim_{N\to\infty} \Pb(G_M)=\lim_{M\to\infty}\lim_{N\to\infty} \Pb(G_M\cap F_\e).
\end{equation}
The centerings in $L_N^{\rm resc}(v)$ and $\widetilde L_N^{\rm resc}(v)$ are the law of large number approximation from $A^{\rm flat}(v)$ and $\widetilde A(v)$ respectively. Define $\mu(m,n)=(\sqrt{m}+\sqrt{n})^2$ (see Proposition~\ref{PropBounds}), then we define
\begin{equation}
\Delta_N(v):=\frac{\mu(E_N(w)-A(v))-\mu(E_N(w)-A^{\rm flat}(v))}{\chi^{-2/3} N^{1/3}}.
\end{equation}
Then
\begin{equation}
\Pb(G_M\cap F_\e)\leq \Pb\Big(\Big\{\max_{|v|\leq M}[\widetilde L_N^{\rm resc}(v)+\Delta_N(v)]\leq s+\e\Big\}\cap F_\e\Big).
\end{equation}
A lower bound on $\Pb(G_M\cap F_\e)$ is obtained with $-\e$ instead of $\e$.

By Assumption A, $\lim_{N\to\infty} \Delta_N(v)={\cal R}(v)=\sqrt{2}\sigma{\cal B}(v)$ weakly. Together with the weak convergence of (\ref{eqAtilde}), we obtain
 \begin{equation}
\begin{aligned}
\lim_{M\to\infty}\lim_{N\to\infty}\Pb(G_M\cap F_\e)&\leq \lim_{M\to\infty}\Pb\Big(\max_{|v|\leq M}[{\cal A}_2(v)-(v-w)^2+{\cal R}(v)]\leq s+\e\Big)\\
&=\Pb\Big(\max_{v\in\R}[{\cal A}_2(v)-(v-w)^2+{\cal R}(v)]\leq s+\e\Big).
\end{aligned}
\end{equation}
The last inequality holds since both the maximum of the Airy$_2$ minus a parabola and of ${\cal R}(v)$ minus a parabola are tight. For the special case of flat initial condition, i.e., when ${\cal R}=0$,
\begin{equation}
\Pb\big(\max_{v\in\R}[{\cal A}_2(v)-(v-w)^2]\leq s\big)\stackrel{(d)}{=}\Pb\big(\max_{v\in\R}[{\cal A}_2(v)-v^2]\leq s\big)=F_{\rm GOE}(2^{2/3} s),
\end{equation}
where we used the fact that the Airy$_2$ process is stationary, and the last equality was proven in~\cite{Jo03b}. This ends the proof of Theorem~\ref{ThmMainLPP}.
\end{proof}

\begin{thm}[Functional slow-decorrelation]\label{ThmFctSlowDec}
Consider any down-right path $\cal L$ passing a.s.\ at a finite-distance from the origin. Let $\widetilde A(v)$ be as in (\ref{eqAtilde}) and let $B(v)$ be the closest point on $\cal L$ to the line from $\widetilde A(v)$ to $E_N(w)$. Consider the rescaled processes (defined for any $v\in\R$ through linear interpolation)
\begin{equation}\label{defRescProcess}
L_N^{\rm resc,B}(v):=\frac{L_{B(v)\to E_N(w)}-\mu(E_N(w)-B(v))}{\chi^{-2/3} N^{1/3}},\quad \mu(m,n)=(\sqrt{m}+\sqrt{n})^2
\end{equation}
as well as $\widetilde L^{\rm resc}_N$ given in (\ref{eq4.10}). Then $L^{\rm resc,B}_N-\widetilde L^{\rm resc}_N$ converges in probability to $0$ in ${\cal C}([-M,M])$ as $N\to\infty$. More precisely, for any $\e,\tilde\e>0$ there is a $N_0$ such that for all $N\geq N_0$,
\begin{equation}
\Pb\left(\max_{|v|\leq M} |L^{\rm resc,B}_N(v)-\widetilde L^{\rm resc}_N(v)|\geq \e\right)\leq \tilde \e.
\end{equation}
\end{thm}
\begin{proof}
The proof is almost identical to the one of Theorem~2.10 in~\cite{CFS16}, see also Theorem~2.15 of~\cite{CLW16} (which is two pages long) and therefore we do not repeat it. Let us just mention the strategy and on the way the inputs which are needed. Using Theorem~\ref{ThmTightness} one knows that the processes along the horizontal lines $\cal L^\pm$ crossing $A(\pm M)$ are tight. One defines the rescaled processes $\widetilde L^{\rm resc,\pm}_N(v)$ to be the analogues of $\widetilde L^{\rm resc}_N(v)$ but with starting points on $\cal L^\pm$, which we call $\widetilde A^\pm(v)$, see Figure~\ref{FigLPPzoom}. Using \emph{tightness} of $\widetilde L^{\rm resc}_N$ (see Theorem~\ref{ThmTightness}) and \emph{one-point slow-decorrelation} (see Theorem~\ref{ThmSlowDecOnePt}) one bounds $\max_{|v|\leq M}|\widetilde L^{\rm resc,\pm}_N(v)-\widetilde L^{\rm resc}_N(v)|$. Finally one needs to control for example the increments of $\widetilde L^{\rm resc,+}_N(v)-L^{\rm resc}_N(v)$. For this one employs use of the subadditivity property of LPP, $L_{\widetilde A^+(v)\to E_N(w)}\geq L_{\widetilde A^+(v)\to A(v)}+L_{A(v)\to E_N(w)}$, and \emph{the bound on the left tail} of $L_{\widetilde A^+(v)\to A(v)}$ provided in Proposition~\ref{PropBounds}.
\end{proof}

A direct consequence of tightness of $\widetilde L_N^{\rm resc}$ and the functional slow-decorrelation result (Theorem~\ref{ThmFctSlowDec}) is the following.
\begin{cor}\label{CorTightness}
Fix any $M\in (0,\infty)$. Then the rescaled LPP process from $\LL$ to $E_N(w)$, $v\mapsto L_N^{\rm resc}(v)$ defined in (\ref{defRescProcess}), is tight in the space of continuous functions on $[-M,M]$, ${\cal C}([-M,M])$. It converges weakly to an Airy$_2$ process $u\mapsto {\cal A}_2(u)$.
\end{cor}

\begin{lem}\label{LemmaBoundOnR}
Define $G_M=\{\max_{|v|\leq M} L_{A(v)\to E_N(w)}\leq a_0 N +a_1 s N^{1/3}\}$ and \mbox{$R_M=\{\max_{|v|> M} L_{A(v)\to E_N(w)}>a_0 N +a_1 s N^{1/3}\}$}, with $a_0=1/\chi$ and $a_1=1/\chi^{2/3}$. Under Assumption~B, there exists a finite $M_0$ such that for any given $M\geq M_0$,
\begin{equation}
\Pb\left(G_M\cap R_M\right) \leq C e^{-c M^2}+Q(M)
\end{equation}
for some constants $C,c>0$ which are uniform in $N$. In particular, for flat initial conditions (where $Q=0$),
\begin{equation}\label{eq4.17}
\Pb(\textrm{the LPP maximizer starts from }A^{\rm flat}(v)\textrm{ with }|v|\leq M)\geq 1-2C e^{-c M^2}.
\end{equation}
\end{lem}
\begin{proof}
For $s\leq -\tfrac14 M^2$, we have
\begin{equation}
\begin{aligned}
\Pb(G_M\cap R_M)&\leq \Pb(G_M) \leq \Pb(L_{(0,0)\to E_N(w)}\leq a_0 N + a_1 s N^{1/3})\\
&\leq C e^{-c |s|^{3/2}}\leq C e^{-c M^2/8},
\end{aligned}
\end{equation}
where we used the lower tail estimate of the point-to-point LPP from Proposition~\ref{PropBounds}.

Thus we consider below any $s\geq -\tfrac14 M^2$. Let us define a set of points $\widehat L$ and we say that
$\widehat L \prec \LL$ if each point in $\LL\cap \{A(v),|v|>M\}$ can be reached by an up-right paths from a point in $\widehat L$.
Then
\begin{equation}\label{eq4.27}
\begin{aligned}
\Pb(G_M\cap R_M) &\leq \Pb(R_M)\leq \Pb\big(\max_{|v|> M} L_{A(v)\to E_N(w)}>a_0 N -\tfrac14 a_1 M^2 N^{1/3}\big)\\
&\leq \Pb(L_{\widehat L\to E_N(w)}>a_0 N -\tfrac14 a_1 M^2 N^{1/3})+ \Pb(\widehat L\not\prec\LL).
\end{aligned}
\end{equation}
Our choice for $\widehat L$ will be such that $\Pb(\widehat L\not\prec\LL)\leq Q(M)$ for all $N$ large enough. To realize it, it is enough to take any $\widehat L$ such that it stays to the left of a parabola close enough to $\Lflat$. In Figure~\ref{FigLPPEstimate} we illustrate $\widehat L$.
\begin{figure}
\begin{center}
\psfrag{n}[lc]{$n$}
\psfrag{m}[lc]{$m$}
\psfrag{Lf}[cb]{$\Lflat$}
\psfrag{L}[cb]{$\LL$}
\psfrag{Ap}[rc]{$\widehat A(M)$}
\psfrag{Am}[rc]{$\widehat A(-M)$}
\psfrag{Apf}[lc]{$A^{\rm flat}(M)$}
\psfrag{Amf}[lc]{$A^{\rm flat}(-M)$}
\psfrag{Cp}[cc]{$\widehat C_-$}
\psfrag{Cm}[cc]{$\widehat C_+$}
\psfrag{E}[lc]{$E_N(w)$}
\includegraphics[height=7cm]{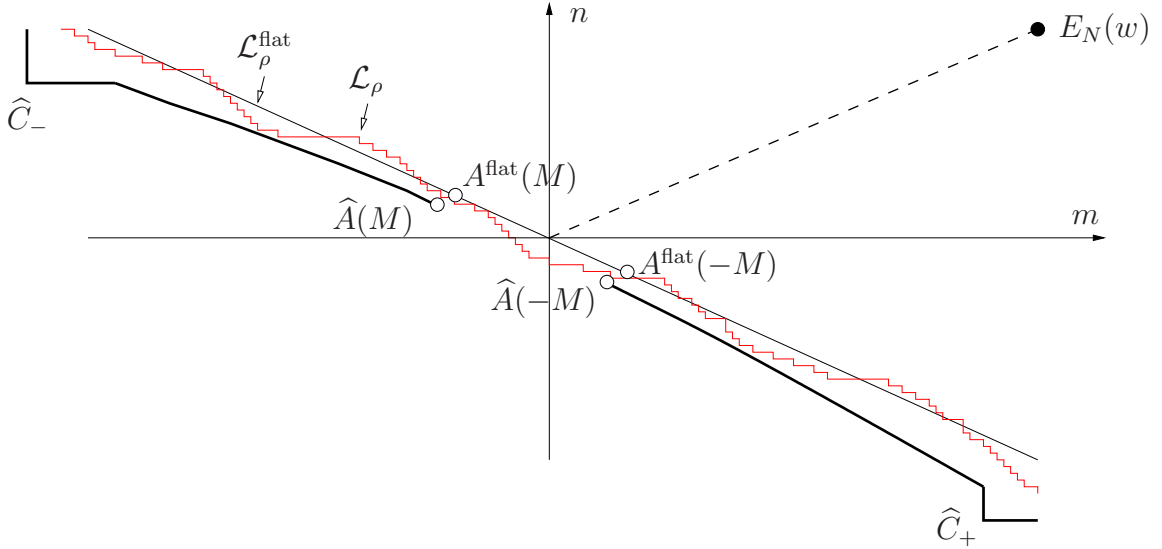}
\caption{The setting used to control the LPP outside the central part. The thick black line is $\widehat L$.}
\label{FigLPPEstimate}
\end{center}
\end{figure}
For a $\delta>0$, we define the points
\begin{equation}
\widehat A(v)=A^{\rm flat}(v)-\delta v^2 N^{1/3} \mathbf{e}_\rho,\quad \mathbf{e}_\rho=((1-\rho)^2,\rho^2),
\end{equation}
the segments ${\cal D}_k=\overline{\widehat A(kM) \widehat A((k+1)M)}$ and $\widetilde {\cal D}_{\ell}=\overline{\widehat A(-\ell M) \widehat A(-(\ell+1)M)}$, and the points $C_+=(-(1+\tfrac{1-\rho}{16}),\tfrac{\rho}{1-\rho}(1-\tfrac{\rho}{16}))N$ and $C_-=(\tfrac{1-\rho}{\rho}(1-\tfrac{1-\rho}{16}),-(1-\tfrac{\rho}{16}))N$. Then, we define
\begin{equation}
\widehat L = C_+ \cup C_- \bigcup_{|v|\geq N^{\nu/3}}{\widehat A(v)}\bigcup_{k=1}^{N^{\nu/3}} {\cal D}_k \bigcup_{\ell=1}^{N^{\nu/3}} \widetilde {\cal D}_{\ell},
\end{equation}
with $\nu\in (0,1/2)$ ($\nu<1/2$ is needed only in the last estimate of this lemma), and the union $A(v)$ is for $v$ up to the $v$ such that $A(v)$ is  reachable by an up-right path from $C_+$ or $C_-$ (there are $\Or(N^{1/3})$ of such $v$). The constant $\delta$ is now chosen small enough such that taking $v_+=\frac{\chi^{1/3}}{2(1-\rho)} N^{1/3}$, which corresponds to $A^{\rm flat}(v_+)=(-N,\tfrac{\rho}{1-\rho}N)$, then $C_+\prec \widehat A(v_+)$, and similarly for side close to $C_-$.

With the $\widehat L$ defined as above, we can apply Assumption~B to bound $\Pb(\widehat L\not\prec\LL)$. It thus remains to get a bound for $\Pb(L_{\widehat L\to E_N(w)}>a_0 N -\tfrac14 a_1 M^2 N^{1/3})$. This can be bounded by
\begin{equation}
\begin{aligned}
&\Pb(L_{C_+\to E_N(w)}>a_0 N -\tfrac{a_1 M^2}{4} N^{1/3}) +\sum_{k=1}^{N^{\nu/3}} \Pb(L_{{\cal D}_k\to E_N(w)}>a_0 N -\tfrac{a_1 M^2}{4} N^{1/3})\\
&+ \Pb(L_{C_-\to E_N(w)}>a_0 N -\tfrac{a_1 M^2}{4} N^{1/3}) +\sum_{\ell=1}^{N^{\nu/3}} \Pb(L_{\widetilde {\cal D}_\ell\to E_N(w)}>a_0 N -\tfrac{a_1 M^2}{4} N^{1/3})\\
&+\sum_{N^{\nu/3}\leq |v|\leq \Or(N^{1/3})}\Pb(L_{\widehat A(v)\to E_N(w)}>a_0 N -\tfrac{a_1 M^2}{4} N^{1/3}).
\end{aligned}
\end{equation}

For the point-to-point estimates we can use the bounds of Proposition~\ref{PropBounds}, which are uniform for the slopes $\eta$ in a bounded set of $(0,\infty)$. To avoid slopes which are close to $0$ or $\infty$, we need to restrict the use of the point-to-point estimates for the LPP from $\widehat A(v)$ and add the LPP from the starting points $C_\pm$ as well.

\smallskip\emph{1st bound.} The points $C_\pm$ are chosen such that from the law of large numbers approximation of $L_{C_\pm \to E_N(w)}$ is less then $a_0 N-N/2$ for any $\rho \in (0,1)$. This means that a deviation of $-\tfrac{a_1 M^2}{4} N^{1/3}$ from $a_0 N$ of $L_{C_+\to E_N(w)}$ corresponds to look at the right tail at a value at least $N/2-\Or(M^2N^{1/3})$. Thus for any given $M$, for all $N$ large enough, Proposition~\ref{PropBounds} implies
\begin{equation}\label{eq4.30}
\Pb(L_{C_+\to E_N(w)}>a_0 N -\tfrac{a_1 M^2}{4} N^{1/3}) \leq C e^{-c N^{2/3}}
\end{equation}
for some constants $C,c$ which depend only on $\rho$. Similarly one has the estimate for $\Pb(L_{C_-\to E_N(w)}>a_0 N -\tfrac14 a_1 M^2 N^{1/3})$.

\smallskip\emph{2nd bound.} In a similar way, using the bound of Proposition~\ref{PropBounds}, for any $N$ large enough,
\begin{equation}
\Pb(L_{\widehat A(v)\to E_N(w)}>a_0 N -\tfrac14 a_1 M^2 N^{1/3}) \leq C e^{-c N^{2\nu/3}}
\end{equation}
for any $v\in [N^{\nu/3},\Or(N^{1/3})]$, and thus
\begin{equation}
\sum_{N^{\nu/3}\leq |v|\leq \Or(N^{1/3})}\Pb(L_{\widehat A(v)\to E_N(w)}>a_0 N -\tfrac14 a_1 M^2 N^{1/3}) \leq C N^{1/3} e^{-c N^{2\nu/3}}\leq C e^{-\tfrac12 c N^{2\nu/3}}
\end{equation}
for $N\gg 1$.

\smallskip\emph{3rd bound.} Finally we need a bound for $\Pb(L_{{\cal D}_k\to E_N(w)}>a_0 N -\tfrac14 a_1 M^2 N^{1/3})$ uniform in $N$, which is summable in $k$ and such that its sum is going to zero as $M\to\infty$. The bound for $\Pb(L_{\widetilde {\cal D}_\ell\to E_N(w)}>a_0 N -\tfrac14 a_1 M^2 N^{1/3})$ is completely analogue and thus we present in details only the first one.

For a given $v$, we define the point $\widehat D(v)$ such that its second coordinate equals the one of $\widehat A(kM)$ and the segment $\overline{\widehat D(v),\widehat A(v)}$ has direction $\mathbf{e}_\rho$. We have
\begin{equation}
\widehat D(v)=A^{\rm flat}(v)-\theta \mathbf{e}_\rho,\quad \theta=\delta (kM)^2 N^{1/3}+\frac{2 (v-kM)N^{2/3}}{\rho\chi^{1/3}}.
\end{equation}

Then, for any $k\geq 1$ and $M$,
\begin{equation}\label{eq4.35}
\begin{aligned}
&\Pb\left(L_{{\cal D}_k\to E_N(w)}>a_0 N -\tfrac{a_1 M^2}{4} N^{1/3}\right) \leq \Pb\left(L_{\widehat A(kM)\to E_N(w)}>a_0 N -\tfrac{3a_1 k^2 M^2}{4} N^{1/3}\right)\\
&+\Pb\left(\max_{kM\leq v\leq (k+1)M} \{L_{\widehat A(v)\to E_N(w)}-L_{\widehat D(v)\to E_N(w)}+\beta N^{2/3}\}\geq \tfrac{a_1 k^2 M^2}{4} N^{1/3}\right)\\
&+\Pb\left(\max_{kM\leq v\leq (k+1)M} \{L_{\widehat D(v)\to E_N(w)}-L_{\widehat A(kM)\to E_N(w)}-\beta N^{2/3}\}\geq \tfrac{a_1 k^2 M^2}{4} N^{1/3}\right),
\end{aligned}
\end{equation}
where $\beta=\tfrac{2(v-kM)}{\rho\chi^{1/3}}-\delta (v^2-(kM)^2)N^{-1/3}$ (which is positive for all $N$ large enough, since $v\in [kM,(k+1)M]$ with $k\in [1,\Or(N^{\nu/3})]$).

\smallskip\emph{Bound on first term of (\ref{eq4.35}).} The law of large numbers estimate of $L_{\widehat A(kM)\to E_N(w)}$ is $a_0 N+N^{1/3}(\delta (kM)^2-a_1(kM-w)^2)$. Thus for any $\delta<\chi^{2/3}/8$ and $M$ large enough, we can use again the point-to-point estimate and obtain
\begin{equation}\label{eq4.36}
\Pb\left(L_{\widehat A(kM)\to E_N(w)}>a_0 N -\tfrac34 a_1 k^2 M^2 N^{1/3}\right) \leq C e^{-c k^2 M^2/8}.
\end{equation}

\smallskip\emph{Bound on second term of (\ref{eq4.35}).} Using $L_{\widehat D(v)\to E_N(w)}\geq L_{\widehat D(v)\to \widehat A(v)}+L_{\widehat A(v)\to E_N(w)}$ we have
\begin{equation}\label{eq4.37}
\begin{aligned}
&\Pb\left(\max_{kM\leq v\leq (k+1)M} \{L_{\widehat A(v)\to E_N(w)}-L_{\widehat D(v)\to E_N(w)}+\beta N^{2/3}\}\geq \tfrac{a_1 k^2 M^2}{4} N^{1/3}\right) \\
&\leq \sum_{kM\leq v\leq (k+1)M} \Pb\left(L_{\widehat A(v)\to E_N(w)}-L_{\widehat D(v)\to E_N(w)}+\beta N^{2/3}\geq \tfrac{a_1 k^2 M^2}{4} N^{1/3}\right)\\
&\leq \sum_{kM\leq v\leq (k+1)M} \Pb\left(L_{\widehat D(v)\to \widehat A(v)}-\beta N^{2/3}\leq -\tfrac{a_1 k^2 M^2}{4} N^{1/3}\right).
\end{aligned}
\end{equation}
Since $L_{\widehat D(v)\to \widehat A(v)}$ centered by $\beta N^{2/3}$ and scaled by $\Or(N^{2/9})$ converges to a $F_{\rm GUE}$ distributed random variable, by the lower tail estimate of Proposition~\ref{PropBounds} we get
\begin{equation}
\Pb(L_{\widehat D(v)\to \widehat A(v)}-\beta N^{2/3}\leq -a_1 k M N^{1/3})\leq C e^{-c k^2 M^2 N^{1/9}}
\end{equation}
for some constants $C,c$ which can be taken independent of $v\in [kM,(k+1)M]$. Since the sum in (\ref{eq4.37}) is over a number of terms $\Or(N^{2/3})$ we get
\begin{equation}\label{eq4.39}
(\ref{eq4.37})\leq C e^{-\frac12 c k^2 M^2 N^{1/9}}
\end{equation}
for all $N$ large enough.

\smallskip\emph{Bound on third term of (\ref{eq4.35}).} For this bound we will employ, between other results, Lemma~\ref{lemmaBoundsRescaledProcesses}. Let us first reformulate what we need to prove in terms of $L^{\rm resc,h}_n$. One looks the picture from the point $E_N(w)$, which becomes the origin. The point $\widehat A(kM)$ as seen from $E_N(w)$ becomes the point $(\gamma^2 n,n)$ and the point $\widehat D(v)$ is $(\gamma^2 n+\beta_1 u(v) n^{2/3},n)$. This means that we need to take
\begin{equation}
\begin{aligned}
&n=\frac{\rho}{1-\rho} N-\frac{2\rho(kM-w)}{\chi^{1/3}}N^{2/3}+\delta \rho^2 (kM)^2 N^{1/3},\\
&\gamma=\frac{1-\rho}{\rho}\Big(1+\frac{kM-w}{\chi^{1/3}}N^{-1/3}+\frac{(kM-w)^2 (3-4\rho)}{2\chi^{2/3}}N^{-2/3}+\Or(N^{-1})\Big),\\
&u(v)=(v-kM)(1+\Or(N^{-2/3})).
\end{aligned}
\end{equation}
We have, in distribution,
\begin{equation}
L_{\widehat D(v)\to E_N(w)}\stackrel{d}{=}L_{(0,0)\to(\gamma^2 n+\beta_1 u(v) n^{2/3},n)}.
\end{equation}
Recall that $\widehat D(kM)=\widehat A(kM)$. Furthermore, the difference between the laws of large numbers of $L_{\widehat D(v)\to E_N(w)}$ and $L_{\widehat A(kM)\to E_N(w)}$ is given by
\begin{equation}
\begin{aligned}
&\beta N^{2/3} - \chi^{-2/3}N^{1/3}\Big[(v-kM)^2(1+\delta \chi^{2/3})+(v-kM)(2w+2kM(1+\delta \chi^{2/3}))\Big]\\
\leq &\beta N^{2/3}-\chi^{-2/3}N^{1/3} u(v)^2(1+\Or(N^{-2/3})),
\end{aligned}
\end{equation}
for all $M$ large enough.

As a consequence, the third term of (\ref{eq4.35}) can be rewritten as
\begin{equation}\label{eq4.40}
\Pb\left(\max_{kM\leq v\leq (k+1)M}\{ L^{\rm resc,h}_n(u(v))-L^{\rm resc,h}_n(0)-u(v)^2 +\Or(n^{-2/3})\}\geq \tfrac14 k^2 M^2\right).
\end{equation}
Applying the upper bound of Lemma~\ref{lemmaBoundsRescaledProcesses} we obtain
\begin{equation}\label{eq4.41}
\begin{aligned}
(\ref{eq4.40})&\leq \Pb(Z^{\rho_+}(\gamma^2 n,n)<0)\\
& + \Pb\left(\max_{u\in I_M} \{B^{\rho_+}_n(u(v))+2\beta_2\kappa u(v)+\Or(n^{-2/3})\}\geq  \tfrac14 k^2 M^2\right),
\end{aligned}
\end{equation}
where $I_M=[0,M(1+\Or(n^{-2/3}))]$. With the choice $\kappa=\e_0 k M$ and, taking $M$ large enough so that we get to use Lemma~\ref{lemmaExitPoint}, we have
\begin{equation}\label{eq4.42}
(\ref{eq4.41})=C e^{-c \e_0^2 k^2 M^2}+\Pb\left(\max_{u\in I_M} \{B^{\rho_+}_n(u)+2\beta_2\kappa u+\Or(n^{-2/3})\}\geq  \tfrac14 k^2 M^2\right)
\end{equation}
We choose $\e_0$ small enough such that for any $M,k\geq 1$, $\max_{u\in I_M}2\beta_2\kappa u+\Or(n^{-2/3})$ is bounded by $\tfrac18 k^2 M^2$ (uniformly for large $n$). Then
\begin{equation}
(\ref{eq4.42})\leq C e^{-c \e_0^2 k^2 M^2} +\Pb\left(\max_{u\in I_M} B^{\rho_+}_n(u)\geq  \tfrac18 k^2 M^2\right).
\end{equation}
In the stationary setting, recall that we defined $\rho_0=\rho_0(\gamma):=1/(1+\gamma)$. By stationarity
\begin{equation}
B^{\rho_+}_n(u)=\frac{1}{\beta_2 n^{1/3}}\sum_{m=1}^{\beta_1 u n^{2/3}} (X_m-(1-\rho_+)^{-1}),
\end{equation}
where $X_1,X_2,\ldots$ are i.i.d.\ random variables ${\rm Exp}(1-\rho_+)$ with $\rho_+=\rho_0+\e_0 k M n^{-1/3}$. Denote by $Y_m=X_m-(1-\rho_+)^{-1}$. Then $T\mapsto Z_T=\sum_{m=1}^T Y_m$ is a martingale. Using the generic maximal inequality for martingale $\Pb(\max_{1\leq t\leq T} Z_t\geq S)\leq \frac{\E(f(Z_T))}{f(S))}$ with $f(x)=e^{\lambda x}$, $\lambda>0$, we have
\begin{equation}\label{eq4.44}
\Pb\left(\max_{u\in I_M} B^{\rho_+}_n(u)\geq  \tfrac18 k^2 M^2\right) \leq \min_{\lambda>0} \frac{(\E(e^{\lambda Y_1}))^T}{e^{\lambda S}} = e^{-S(1-\rho_+)+T\ln[1+(1-\rho_+) S/T]},
\end{equation}
with $S=\tfrac18 k^2 M^2\beta_2 n^{1/3}$ and $T=\beta_1 u(M) n^{2/3}=2M\beta_1 n^{2/3}(1+\Or(n^{-1/3}))$. A computation then leads to
\begin{equation}\label{eq4.45}
(\ref{eq4.44})= \exp\left(-\frac{k^4 M^3}{512}(1+\Or(k^2 n^{-1/3}))\right).
\end{equation}
Remember that the range of $k$ is from $1$ to $\Or(n^{\nu/3})$. Thus the error term is in the worst case $\Or(n^{(2\nu-1)/3})$. Therefore we can now set the value of $\nu$ to be any number in $(0,1/2)$, e.g., $\nu=1/3$. With this choice, for $n$ large enough, the error term is not larger than $1$ and thus for any $k,M$,
\begin{equation}\label{eq4.46}
(\ref{eq4.45})\leq \exp(-c k^2 M^2).
\end{equation}

Summing up the estimates we have
\begin{equation}
\sum_{k\geq 1} \Pb\left(L_{{\cal D}_k\to E_N(w)}>a_0 N -\tfrac{a_1 M^2}{4} N^{1/3}\right)\leq \sum_{k\geq 1}\Big( (\ref{eq4.36})+(\ref{eq4.39}) + (\ref{eq4.46}) \Big)\leq C e^{-c M^2}
\end{equation}
for all $N$ large enough. Here the constants $C,c$ are uniform in $N$ and $M$.

Finally we need to prove (\ref{eq4.17}). Notice that for flat initial condition we have $Q=0$ and thus
\begin{equation}
\begin{aligned}
&\Pb(\textrm{the LPP maximizer starts from }A^{\rm flat}(v)\textrm{ with }|v|\leq M)\\
=& \Pb\Big(\max_{|v|\leq M} L_{A(v)\to E_N(w)} >\max_{|v|> M} L_{A(v)\to E_N(w)}\Big)\geq \Pb(G_M^c\cap R_M^c)\\
\geq & 1-\Pb(G_M) - \Pb(R_M),
\end{aligned}
\end{equation}
for any choice of $s$. With the choice $s=-M^2/4$, the bounds obtained above lead to the claimed result.
\end{proof}

\appendix
\section{Bounds on point-to-point LPP}
In the proof we use known results for the point-to-point LPP with exponential random variables, which we recall here.
\begin{prop}\label{PropBounds}
For $\eta\in(0,\infty)$ define $\mu=(\sqrt{\eta \ell}+\sqrt{\ell})^2$, $\sigma=\eta^{-1/6}(1+\sqrt{\eta})^{4/3}$, and the rescaled random variable
\begin{equation}
L^{\rm res}_\ell:=\frac{L_{(0,0)\to(\eta\ell,\ell)}-\mu}{\sigma\ell^{1/3}}.
\end{equation}
(a) Limit law
\begin{equation}
\lim_{\ell\to\infty} \Pb(L^{\rm res}_\ell\leq s) = F_{\rm GUE}(s),
\end{equation}
with $F_{\rm GUE}$ the GUE Tracy-Widom distribution function.\\
(b) Bound on upper tail: there exist constants $s_0,\ell_0,C,c$ such that
\begin{equation}
\Pb(L^{\rm res}_\ell\geq s)\leq C e^{-c s}
\end{equation}
for all $\ell\geq \ell_0$ and $s\geq s_0$.\\
(c) Bound on lower tail: there exist constants $s_0,\ell_0,C,c$ such that
\begin{equation}
\Pb(L^{\rm res}_\ell\leq s)\leq C e^{-c |s|^{3/2}}
\end{equation}
for all $\ell\geq \ell_0$ and $s\leq -s_0$.
\end{prop}
The constants $C,c$ can be chosen uniformly for $\eta$ in a bounded set. (a) was proven in Theorem~1.6 of\cite{Jo00b}. Using the relation with the Laguerre ensemble of random matrices (Proposition~6.1 of~\cite{BBP06}), or to TASEP described above, the distribution is given by a Fredholm determinant. An exponential decay of its kernel leads directly to (b). See e.g.\ Proposition~4.2 of~\cite{FN13} or Lemma~1 of~\cite{BFP12} for an explicit statement. (c) was proven in~\cite{BFP12} (Proposition~3 together with (56)). In the present language it is reported in Proposition~4.3 of~\cite{FN13} as well.

\section{One-point slow-decorrelation theorem}
Here we state one-point slow-decorrelation theorem in the setting of point-to-point LPP with homogeneous waiting times, since it is what we employ in our paper. The statement of Theorem~2.1 in~\cite{CFP10b} is for more generic LPP problems. The application to finitely many points is straightforward using union bound and it was already used for instance in~\cite{CFP10a,BFP09}.

\begin{thm}[One-point slow-decorrelation]\label{ThmSlowDecOnePt}
Let $p\in\R_+^2$ be a direction. Assume that there exist constants is a $\mu=\mu(p)$, a distribution $D$, an $\alpha\in(0,1)$ and $\nu\in (0,1)$, such that
\begin{equation}
\frac{L_{(0,0)\to [p \ell]}-\mu \ell}{\ell^{\alpha}}\Rightarrow D,\textrm{ as }t\textrm{ goes to infinity}.
\end{equation}
Then, for any $\e>0$,
\begin{equation}
\lim_{\ell\to\infty}\Pb\left(|L_{(0,0)\to [p (\ell+ \ell^\nu)]}-L_{(0,0)\to [p \ell]}-\mu \ell^\nu|\geq \e\ell^\alpha\right)=0.
\end{equation}
\end{thm}
The assumptions for the model considered in this paper are satisfied with \mbox{$p=(\eta,1)$}, $\mu=(1+\sqrt{\eta})^2$, $\alpha=1/3$, and $D$ is $F_{\rm GUE}$ (up to a scaling), see Proposition~\ref{PropBounds}.


\end{document}